\documentclass{svmult}

\usepackage{makeidx}         
\usepackage{graphicx}        
\usepackage{multicol}        
\usepackage[bottom]{footmisc}
\usepackage{natbib}
\usepackage{url}

\renewcommand{\PackageWarningNoLine}[2]{}
\usepackage{amsmath}
\usepackage{amsfonts}
\usepackage{bm}

\usepackage{tikz}
\usepackage{epstopdf}

\begin{document}

\title*{Domain decomposition methods \\ with overlapping subdomains \\ for time-dependent problems}
\titlerunning{DDM with overlapping subdomains for time-dependent problems} 
\author{Petr Vabishchevich and Petr Zakharov}

\institute{Petr Vabishchevich \at Nuclear Safety Institute of RAS, Moscow, Russia, \email{vabishchevich@gmail.com}
\and Petr Zakharov \at North-Eastern Federal University, Yakutsk, Russia, \email{zapetch@gmail.com}}
\maketitle

\section{Introduction}
\label{vabishchevich_contrib-sec:1}

Domain decomposition methods are used for the numerical solution of boundary value problems 
for partial differential equations on parallel computers.  
In the theory of domain decomposition methods, modern studies are most fully presented for stationary 
problems \cite{QuarteroniValli1999,ToselliWidlund2005}.
Computational algorithms with and without overlap of subdomains are applied
in synchronous (sequential) and asynchronous (parallel) methods.  

Domain decomposition methods for unsteady problems are based on two basic
approaches \cite{SamarskiiMatusVabischevich2002}. 
\begin{enumerate}
 \item 
For the numerical solution of time-dependent problems, we use 
the standard implicit approximation in time.    
Domain decomposition methods are applied to solve the discrete problem at the new time level.  
The number of iterations in optimal iterative methods 
for domain decomposition does not depend on  steps of discretization in time and space.
 \item 
To solve unsteady problems, iteration-free domain decomposition algorithms are developed.  
We construct a special scheme of splitting into subdomains  (regionally additive schemes).
\end{enumerate} 
A domain decomposition scheme is defined by a decomposition of the computational domain and by specifying a splitting of the problem operator.  
To construct decomposition operators, it is convenient to use the partition of unity for the computational domain.

In DD methods with overlap, we introduce a function associated with each subdomain, and this function takes value between zero and one.  
In the extreme case, the width of overlap of subdomains is equal to the step of discretization in space.  
In this case, regionally additive schemes can be interpreted as non-overlapping domain decomposition schemes, where data exchange is achieved by setting proper boundary conditions 
for each subdomain.  

Domain decomposition methods for unsteady problems include the following steps:
\begin{itemize}
 \item  Decomposition of a domain;
 \item  Constructing operators of decomposition;
 \item  Design of a splitting scheme;
 \item  A study of convergence;
 \item  Computational implementation.
\end{itemize}
These basic questions are discussed
in this paper using a boundary value problem for the second-order parabolic equation as an example.

\section{Standard approximation} 

Assume that in a bounded domain $\Omega$, a unknown function $u({\bm x},t)$
satisfies the following equation:
\begin{equation}\label{vabishchevich_contrib-1}
   \frac{\partial u}{\partial t} 
   - \sum_{\alpha =1}^{m}
   \frac{\partial }{\partial x_\alpha} 
   \left ( k({\bm x})  \frac{\partial u}{\partial x_\alpha} \right ) = f({\bm x},t),
   \quad {\bm x}\in \Omega,
   \quad 0 < t \leq  T, 
\end{equation} 
where $k({\bm x}) \geq \kappa > 0, \  {\bm x}\in \Omega$.
Homogeneous Dirichlet boundary conditions are applied:
\begin{equation}\label{vabishchevich_contrib-2}
   u({\bm x},t) = 0,
   \quad {\bm x}\in \partial \Omega,
   \quad 0 < t \leq T.
\end{equation} 
The initial condition is
\begin{equation}\label{vabishchevich_contrib-3}
   u({\bm x},0) = u^0({\bm x}),
   \quad {\bm x}\in \Omega. 
\end{equation} 

Let $(\cdot,\cdot), \|\cdot\|$ be the scalar product and the norm in $L_2(\Omega)$, respectively:
\[
(u,v) = \int_\Omega u({\bm x})v({\bm x}) d {\bm x},
\quad \|u\| = (u,u)^{1/2}.
\]
A symmetric, positive definite, bilinear form $d(u,v)$ such that
\[
d(u,v) = d(v,u), 
\quad  d(u,u) \geq \delta \|u\|^2,
\quad  \delta > 0,
\]
is associated with a Hilbert space $H_d$  equipped with the following 
scalar product and norm:
\[
(u,v)_d  = d(u,v), \quad  \|u\|_d  = (d(u,u))^{1/2}.
\]

Suppose $t = t^n  = n\tau , \ n = 0,1,...$, where $\tau > 0$ is a constant time step.
A finite-dimensional space of finite elements is denoted by ${\cal V}^h$,
and $u^n = u(\bm x, t^n)\ (u^n \in {\cal V}^h)$ stands for the approximate solution at the time level $t = t^n$.
The boundary value problem (\ref{vabishchevich_contrib-1})--(\ref{vabishchevich_contrib-3}) is treated in the variational form:
\begin{equation}\label{vabishchevich_contrib-4}
 \left (\frac{d u}{d t} , v \right ) + a(u,v) = (f,v),
 \quad \forall v \in H^1_0 (\Omega), 
 \quad 0 < t \leq T, 
\end{equation} 
\begin{equation}\label{vabishchevich_contrib-5}
 (u(0), v) = (u^0, v) ,  
 \quad \forall v \in H^1_0 (\Omega) , 
\end{equation} 
where
\[
 a(u,v) = \int_{\Omega} k({\bm x}) \ {\rm grad} \, u \ {\rm grad} \, v \ d {\bm x} .
\] 

We study the projection-difference scheme (schemes with weights) for (\ref{vabishchevich_contrib-4}), (\ref{vabishchevich_contrib-5}):
\begin{equation}\label{vabishchevich_contrib-6}
 \left (\frac{y^{n+1} - y^{n}}{\tau},v \right ) + 
 a(\sigma y^{n+1} + (1-\sigma) y^n,v)= 
 (f(\sigma t^{n+1} + (1-\sigma) t^n),v), 
\end{equation} 
\begin{equation}\label{vabishchevich_contrib-7}
   (y^0, v) = (u^0, v) , 
  \quad \forall v \in {\cal V}^h, \quad n=1,2,...  , 
\end{equation} 
where $\sigma $ is a number (weight). 
If $\sigma = 0$, then the scheme (\ref{vabishchevich_contrib-6}), (\ref{vabishchevich_contrib-7}) is the explicit (Euler forward-time) scheme; 
for $\sigma = 1$, we obtain the fully implicit (Euler backward-time) scheme; 
and $\sigma = 0.5$ yields the averaged (the so-called Crank--Nicolson) scheme. 
The condition
\[
(v, v)  + \left ( \sigma - \frac{1}{2} \right ) \ \tau \ a(v, v) \geq 0,
\quad \forall v \in {\cal V}^h 
\]
is necessary and sufficient for the stability of the scheme in the space $H_a$
\cite{Samarskii1989,SamarskiiMatusVabischevich2002}.

\section{Decomposition operators} 

To construct a domain decomposition scheme, we introduce the partition of unity for the computational domain  $\Omega$ \cite{Laevsky1987}. 
Assume that the domain $\Omega$  consists of $p$  separate subdomains:  
\[
  \Omega =\Omega _1\cup \Omega _2\cup ...\cup \Omega _p.
\]
Individual subdomains may be overlaped.  
With an individual subdomain $\Omega_{\alpha}, \ \alpha = 1,2,...,p$
we associate the function $\eta_{\alpha}({\bm x}), \ \alpha = 1,2,...,p$ such that  
\[
  \eta_{\alpha}({\bm x}) = \left \{
   \begin{array}{cc}
     > 0, &  {\bm x} \in \Omega_{\alpha},\\
     0, &  {\bm x} \notin  \Omega_{\alpha}, \\
   \end{array}
  \right .
  \quad \alpha = 1,2,...,p ,  
\]
where
\[
  \sum_{\alpha =1}^{p} \eta_{\alpha}({\bm x}) = 1,
  \quad {\bm x} \in \Omega .
\]

For problem (\ref{vabishchevich_contrib-4}), (\ref{vabishchevich_contrib-5}), we have
\[
 a(u,v) = \sum_{\alpha =1}^{p} a_\alpha (u,v) ,
 \quad (f,v) = \sum_{\alpha =1}^{p} (f_\alpha,v) . 
\] 
Here
\[
 (f_\alpha,v) = \int_{\Omega} \eta_{\alpha}({\bm x}) f({\bm x}, t)  v \ d {\bm x} 
\] 
and (the standard decomposition)
\[
 a_\alpha (u,v) = \int_{\Omega} \eta_{\alpha}({\bm x}) k({\bm x}) \ {\rm grad} \, u \ {\rm grad} \, v \ d {\bm x} ,
  \quad \alpha = 1,2,...,p .
\] 

Other variants of domain decompositon operators \cite{Vabischevich1989} are associated with
the following forms:
\[
 a_\alpha (u,v) = \int_{\Omega} k({\bm x}) \ {\rm grad} \, u \ {\rm grad} \,  (\eta_{\alpha}({\bm x}) v) \ d {\bm x} ,
\] 
\[
 a_\alpha (u,v) = \int_{\Omega} k({\bm x}) \ {\rm grad} \, (\eta_{\alpha}({\bm x}) u) \ {\rm grad} \, v \ d {\bm x} ,
  \quad \alpha = 1,2,...,p .
\] 

Let us investigate the corresponding operator-splitting schemes. 
From problem (\ref{vabishchevich_contrib-4}), (\ref{vabishchevich_contrib-5}), we can go to the Cauchy problem
for the evolutionary equation of first order (matrix form \cite{Thomee2006}):
\begin{equation}\label{vabishchevich_contrib-8}
  B \frac{d  y}{d  t} + A y = \varphi (t),    
  \quad 0 < t \leq T , 
\end{equation} 
\begin{equation}\label{vabishchevich_contrib-9}
  y(0) = y^0. 
\end{equation} 
Here the mass matrix $B = B^* > 0$,
and the stiffness matrix $A = A^* > 0$.

For (\ref{vabishchevich_contrib-8}), (\ref{vabishchevich_contrib-9}), we have the following operator splitting:
\[
 A = \sum_{\alpha =1}^{p} A_\alpha,
 \quad \varphi =   \sum_{\alpha =1}^{p} \varphi_\alpha
\] 
with  (the standard decomposition)
\[
 A_\alpha = A_\alpha^* \geq 0,
 \quad \alpha = 1,2,...,p .
\] 
We arrive to the symmetrized equation:
\[
  \frac{d  w}{d  t} + \widetilde{A} w = \widetilde{\varphi} (t),    
  \quad 0 < t \leq T ,
\]
where
\[
 w = B^{1/2} v,  
 \quad \widetilde{A} = B^{-1/2} A B^{-1/2},
 \quad \widetilde{\varphi} = B^{-1/2} \varphi ,
\] 
\[
  \widetilde{A} = \sum_{\alpha =1}^{p} \widetilde{A}_\alpha,
  \quad  \widetilde{A}_\alpha = \widetilde{A}_\alpha^* = B^{-1/2} A_\alpha  B^{-1/2} \geq 0,
  \quad \alpha = 1,2,...,p .
\]
Now we can employ general results of the stability  (correctness) theory  
for operator-difference schemes \cite{Samarskii1989,SamarskiiMatusVabischevich2002}.

\section{Splitting schemes} 

The investigation of domain decomposition schemes for time-dependent problems is based
on consideration of the relevant splitting schemes \cite{Vabischevich2013}.
Here we highlight the case of the two-component splitting ($ p = 2 $).
In this case, we can focus on the following methods:
\begin{itemize}
 \item the Douglas-Rachford scheme,
 \item the Peaceman-Rachford scheme,
 \item Factorized schemes,
 \item Symmetric scheme of componentwise splitting.
\end{itemize} 

In particular, the Douglas-Rachford scheme may be written as:
\[
 \left (\frac{u^{n+1/2} - u^{n}}{\tau},v \right ) + 
 a_1(u^{n+1/2},v) +  a_2(u^{n},v) = 
 (f^{n+1},v),
\]
\[
 \left (\frac{u^{n+1} - u^{n}}{\tau},v \right ) + 
 a_1(u^{n+1/2},v) +  a_2(u^{n+1},v) = 
 (f^{n+1},v),
 \quad \forall v \in {\cal V}^h .
\]
The problem in the subdomain (explicit-implicit scheme) is formulated in the form:
\[
  (u^{n+1/2}, v) + \tau  a_1(u^{n+1/2},v) = (u^{n}, v) - \tau  a_2(u^{n},v) + \tau (f^{n+1},v),
\]
\[
  (u^{n+1}, v) + \tau  a_2(u^{n+1},v) =  (u^{n}, v) - \tau  a_1(u^{n+1.2},v) + \tau (f^{n+1},v).
\]
In the domain  $\Omega \setminus \bar{\Omega}_1$ we have the explicit scheme for computing $u^{n+1/2}$. 
It is sufficient to perform computations for  $\partial \Omega_1 \cap \Omega$. 
Next, we solve the problem in the subdomain $\Omega_1$ to find $u^{n+1/2}$. 
The calculation of $u^{n+1}$ are performed similarly.

In the more general case, we focus on factorized schemes with weights:
\begin{equation}\label{vabishchevich_contrib-10}
 \left (\frac{u^{n+1/2} - u^{n}}{\tau},v \right ) + 
 a_1(\sigma u^{n+1/2} + (1-\sigma) u^{n},v) +  a_2(u^{n},v) = 
 (f^{n+\sigma} ,v), 
\end{equation} 
\begin{equation}\label{vabishchevich_contrib-11}
\begin{split}
 & \left (\frac{u^{n+1} - u^{n}}{\tau},v \right )  + 
 a_1(\sigma u^{n+1/2} + (1-\sigma) u^{n},v) \\
 & +  a_2(\sigma u^{n+1} + (1-\sigma) u^{n},v)  = 
 (f^{n+\sigma} ,v), \quad \forall v \in {\cal V}^h, \quad n=1,2,...  ,
\end{split}
\end{equation} 
where $f^{n+\sigma} = f(\sigma t^{n+1} + (1-\sigma)t^n)$.
For $\sigma = 1/2$, we obtain the Peaceman-Rachford scheme, whereas at
$\sigma = 1$ we have the Douglas-Rachford scheme.

The operator (matrix) form of the factorized scheme (\ref{vabishchevich_contrib-10}), (\ref{vabishchevich_contrib-11})   seems like this:
\begin{equation}\label{vabishchevich_contrib-12}
  (B + \sigma \tau A_1) B^{-1}  (B + \sigma \tau A_2) 
  \frac{y^{n+1} - y^{n}}{\tau }
  +  A y^{n} = \varphi^n .
\end{equation} 

\begin{theorem} 
The factorized regionally additive difference scheme (\ref{vabishchevich_contrib-12}) 
with $\sigma \geq 1/2$ is unconditionally stable.
The following estimate for stability takes place:
\begin{equation}\label{vabishchevich_contrib-13}
 \|(B + \sigma \tau A_2) y^{n+1}\|_{B^{-1}} \leq 
 \|(B + \sigma \tau A_2) y^{n}\|_{B^{-1}} + \tau \|\varphi^n\|_{B^{-1}}  .
\end{equation} 
\end{theorem} 

\begin{proof} 
Taking into account the previously introduced notations the scheme (\ref{vabishchevich_contrib-12}) can be written as 
\[
  (E + \sigma \tau \widetilde{A}_1) (E + \sigma \tau \widetilde{A}_2) 
  \frac{v^{n+1} - v^{n}}{\tau }
  +  \widetilde{A} v^{n} = \widetilde{\varphi}^n , 
\] 
where $v^{n} = B^{-1/2} y^n, \ \widetilde{\varphi}^n =  B^{-1/2}\varphi^n$.
This scheme is unconditionally stable \cite{Vabischevich2013} for  $\sigma \geq 1/2$ and the following estimate holds for the solution
\[
  \|(E + \sigma \tau \widetilde{A}_2) v^{n+1}\| \leq 
 \|(E + \sigma \tau \widetilde{A}_2) v^{n}\| + \tau \|\widetilde{\varphi}^n\| .
\] 
Hence, we obtain the required estimate (\ref{vabishchevich_contrib-13}).
\end{proof} 

For multicomponent splitting, the basic classes of additive schemes \cite{Marchuk1990,Vabischevich2013}
are the following:
\begin{itemize}
 \item Schemes of componentwise splitting, 
 \item Additively averaged schemes of summarized approximation,
 \item Regularized additive schemes,
 \item Vector additive schemes.
\end{itemize} 

\section{Numerical tests} 

As a test problem, we consider the differential problem  (\ref{vabishchevich_contrib-1})--(\ref{vabishchevich_contrib-3}) 
for
\[
\Omega = \{ x \ | \ 0 < x < 1  \}, \quad 0 \leq t \leq T = 2^{-4},
\]
with the constant coefficient $k(x)=1$, the homogeneous right-hand side  ($f(x,t)=0$) and the exact solution 
\[
u(x, t) = \exp(-\pi^2 t) \sin(\pi x). 
\]
The problem is solved on a uniform mesh in time and space with size-steps $h, N h = 1$ и $\tau, M \tau = T$.
The computational domain is divided into intervals of length $H$ and each interval is divided into two subdomains with width of overlap  $q$.
The function $\eta_\alpha, \alpha=1,2$  is defined as shown in Fig.~\ref{vabishchevich_contrib-fig:1}.
Errors of approximate solution are estimated to be ${\displaystyle \varepsilon = \max_{n} \|y^n(x) - u(x,t^n)\|}$.

\begin{figure}
\begin{tikzpicture}[scale=0.7]

	\foreach \x in {0,...,15}
		\draw (\x,0) circle (0.1);

	\foreach \x in {0,1,9}
		\draw[dashed] (\x,1) -- (\x+1,1);
	\foreach \x in {2,10}
		\draw[dashed] (\x,1) -- (\x+3,0);
	\foreach \x in {2,10}
		\draw[dashed] (\x,1) -- (\x+3,0);
	\foreach \x in {6}
		\draw[dashed] (\x,0) -- (\x+3,1);
	\foreach \x in {5,13,14}
		\draw[dashed] (\x,0) -- (\x+1,0);
		
	\draw (-0.5, 1) node {$\eta_1$};
	\draw (1, 0.5) node {$\Omega_1$};
	\draw (9.5, 0.5) node {$\Omega_1$};

	\foreach \x in {0,1,9}
		\draw (\x,0) -- (\x+1,0);
	\foreach \x in {2,10}
		\draw (\x,0) -- (\x+3,1);
	\foreach \x in {2,10}
		\draw (\x,0) -- (\x+3,1);
	\foreach \x in {6}
		\draw (\x,1) -- (\x+3,0);
	\foreach \x in {5,13,14}
		\draw (\x,1) -- (\x+1,1);
		
	\draw (15.5, 1) node {$\eta_2$};
	\draw (5.5, 0.5) node {$\Omega_2$};
	\draw (14, 0.5) node {$\Omega_2$};
	
	\draw[dotted] (0,-1.5) -- (0, 0);
	\draw[dotted] (7.5,-1.5) -- (7.5, 0.5);
	\draw [|<->|] (0,-1.5) -- (7.5,-1.5) node [above,midway]{$H$};

	\draw[dotted] (10,-1.5) -- (10, 0);
	\draw[dotted] (13,-1.5) -- (13, 0);
	\draw [|<->|] (10,-1.5) -- (13,-1.5) node [above,midway]{$q$};

	\draw[dotted] (14,-1.5) -- (14, 0);
	\draw[dotted] (15,-1.5) -- (15, 0);
	\draw [|<->|] (14,-1.5) -- (15,-1.5) node [above,midway]{$h$};
\end{tikzpicture}
\caption{Decomposition of one-dimensional domain into two subdomains with overlap.}
\label{vabishchevich_contrib-fig:1}
\end{figure}
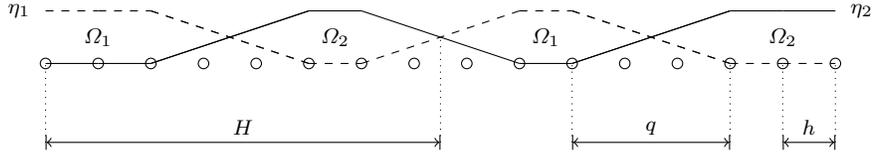

To study the dependence of error on time step we conduct numerical experiments for   $\tau=2^{-4 - \gamma/4},\ \gamma=0,1,...,48$ with $h=2^{-10},\ H = 2^{-1},\ q = h$. 
Fig.~\ref{vabishchevich_contrib-fig:2} presents errors of the implicit scheme, 
the Crank-Nicolson scheme and the factorizied decomposition schemes  
(\ref{vabishchevich_contrib-10}), (\ref{vabishchevich_contrib-11}) for $\sigma = 1$ and $\sigma = 1/2$.
For the errors of decomposition schemes we observe asymptotic behavior  $\mathcal{O}(\tau^2)$ 
for the large $\tau$ and $\mathcal{O}(\tau)$ for the small $\tau$. 

We perform the study of dependence of error on  mesh size using the minimal overlap $q = h$ for $h=2^{-\gamma/4}, \ \gamma=4,5,...,52$ and  $\tau=2^{-10},\ H = 2^{-1}$.
For the implicit scheme, when we decrease the step-size in space, the term $\mathcal{O}(h^2)$  dominates, then the term   $\mathcal{O}(\tau)$ dominates (Fig.~\ref{vabishchevich_contrib-fig:3}).
The asymptotic error of decomposition schemes is close to $\mathcal{O}(h^{-1})$.

\begin{figure}[h!]
\begin{minipage}{0.5\linewidth }
\includegraphics[width= \linewidth]{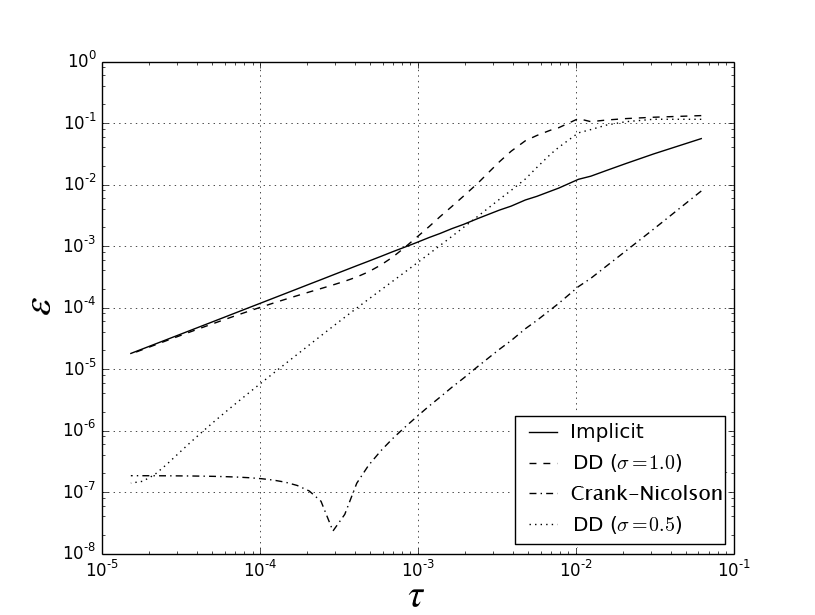}
\caption{Dependence on time step.}
\label{vabishchevich_contrib-fig:2}
\end{minipage}
\begin{minipage}{0.5\linewidth }
\includegraphics[width= \linewidth]{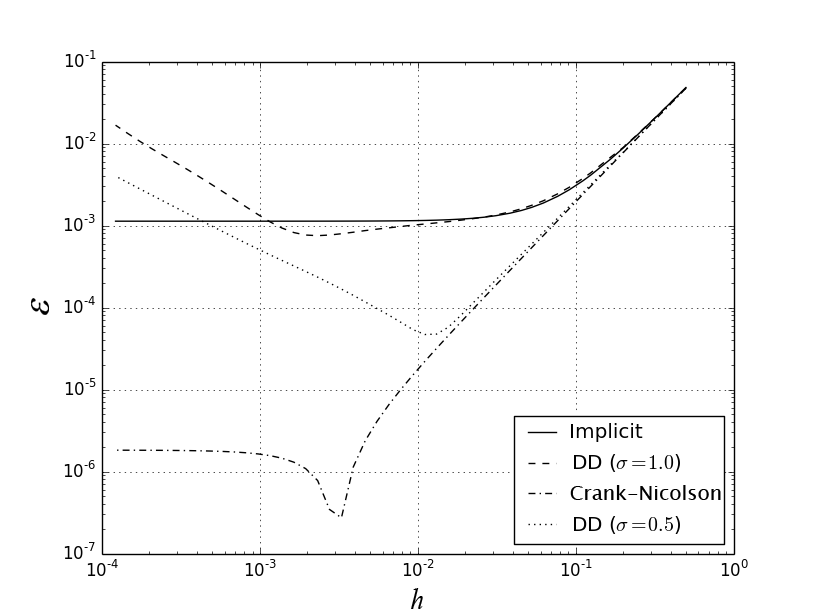}
\caption{Dependence on mesh size.}
\label{vabishchevich_contrib-fig:3}
\end{minipage}
\end{figure} 
\vspace{-5mm} 
The dependence of error on the size of subdomains is shown in Fig.~\ref{vabishchevich_contrib-fig:4}.
Experiments are conducted for  $H=2^{-\gamma/2}, \gamma=0,1,...,14$ with  $h=2^{-10}, \tau=2^{-10}, q=h$. 
Numerical experiments to study the dependence of error on the wight of overlap of subdomains are conducted  $q=(2^{\gamma/4 +1} - 1) h, \ \gamma=4,5,...,36$ with $h=2^{-11}, \tau=2^{-9}, H=2^{-1}$ and presented in Fig.~\ref{vabishchevich_contrib-fig:5}. 

\begin{figure}[h!]
\begin{minipage}{0.5\linewidth }
\includegraphics[width=\linewidth]{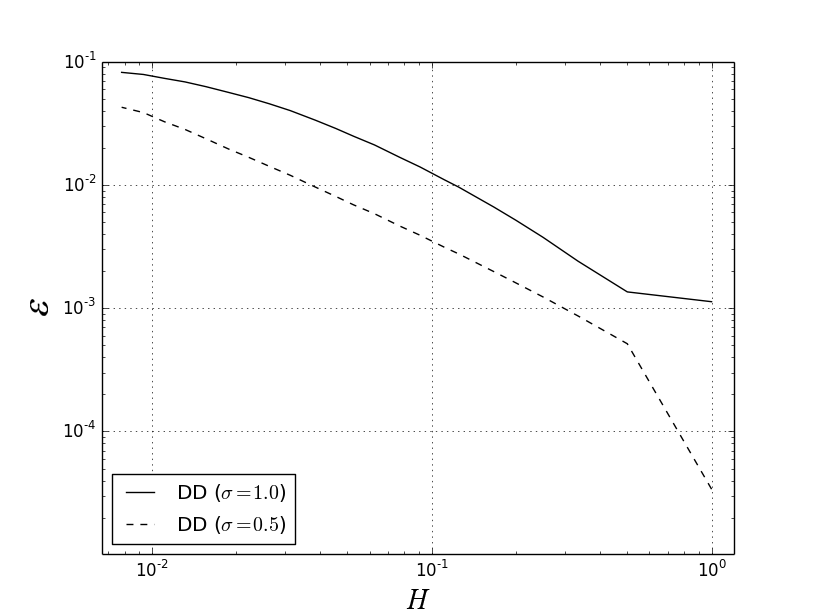}
\caption{Dependence on size of subdomains.}
\label{vabishchevich_contrib-fig:4}
\end{minipage}
\begin{minipage}{0.5\linewidth }
\includegraphics[width=\linewidth]{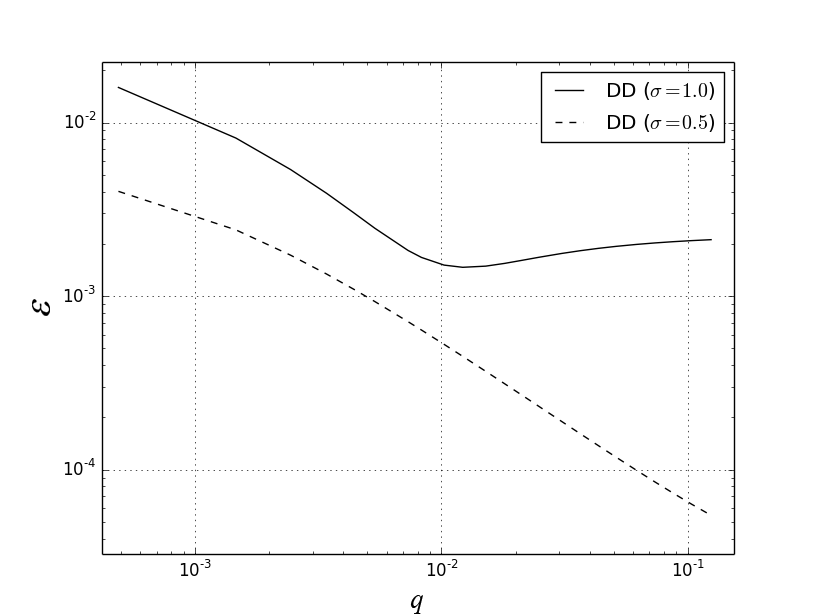}
\caption{Dependence on width of overlap of subdomains.}
\label{vabishchevich_contrib-fig:5}
\end{minipage}
\end{figure} 
\vspace{-10mm} 


\end{document}